\documentclass{article}

\usepackage[T1]{fontenc}
\usepackage[english]{babel}
\usepackage{pdfpages}
\usepackage[a4paper]{geometry}
\usepackage{graphicx}
\usepackage{amsmath,amsthm,amssymb}
\usepackage{wasysym}
\usepackage{microtype}
\usepackage[colorlinks=false, pdfborder={0 0 0}]{hyperref}
\usepackage[capitalise]{cleveref}

\title{Full box spaces of free groups}
\author{Thiebout Delabie \footnote{\texttt{Supported by grant 200021\_163417 of the Swiss National Science Foundation}}}
\date{\today}

\newtheorem{theorem}{Theorem}[section]
\newtheorem{lemma}[theorem]{Lemma}
\newtheorem{proposition}[theorem]{Proposition}
\newtheorem{corollary}[theorem]{Corollary}
\newtheorem{definition}[theorem]{Definition}

\newcommand{\diam}[2][]{\operatorname{diam}_{#1}\left(#2\right)}
\newcommand{\Span}[1]{\operatorname{span}\left(#1\right)}

\newcommand{\Aut}{\operatorname{Aut}}
\newcommand{\Id}{\operatorname{Id}}
\renewcommand{\Im}{\operatorname{Im}}
\newcommand{\dom}{\operatorname{dom}}
\newcommand{\Alln}{$n\ge 3$}
\newcommand{\AllN}{$n\in\N$}
\newcommand{\N}{\mathbb{N}}
\newcommand{\Z}{\mathbb{Z}}

\newcommand{\R}{\mathbb{R}}

\newcommand{\Boxf}{\Box_f}
\newcommand{\norm}[1]{\left\|#1\right\|}
\newcommand{\mingen}{minimal generating set}
\renewcommand{\O}{\mathcal{O}}

\begin{document}

\maketitle

\begin{abstract}
In this paper we investigate full box spaces and coarse equivalences between them. We do this in two parts. In part one we compare the full box spaces of free groups on different numbers of generators. In particular the full box space of a free group $F_k$ is not coarsely equivalent to the full box space of a free group $F_d$, if $d\ge 8k+10$. In part two we compare $\Boxf\Z^n$ to the full box spaces of $2$-generated groups. In particular we prove that the full box space of $\Z^n$ is not coarsely equivalent to the full box space of any $2$-generated group, if \Alln.
\end{abstract}

\section{Introduction}
Given a finitely generated group $G$ we can consider a collection of finite index normal subgroups $(N_i)_i$, and we can create the metrized disjoint union of all $G/N_i$. The metric of $\coprod G/N_i$ is defined as follows: $d(x,y) = d_{G/N_i}(x,y)$ if $x,y\in G/N_i$ and $d(x,y)=\diam{G/N_i}+\diam{G/N_j}$ if $x\in G/N_i$ and $y\in G/N_j$ with $i\not= j$. Note that $G/N_i$ is endowed with the word metric coming from a fixed finite generating set $S$ in $G$.
\\
If $(N_i)_i$ is a decreasing sequence of finite index normal subgroups of $G$ with trivial intersection, then the metrized disjoint union $\coprod G/N_i$ is called a box space of $G$, denoted by $\Box_{N_i}G$. Note that this only exists if $G$ is residually finite.
\\
In this paper we will mainly be concerned with the full box space $\Boxf G$ of $G$, which is the metrized disjoint union of all finite quotients of $G$. We will study these spaces up to coarse equivalence:
\begin{definition}
Let $(X,d_X)$ and $(Y,d_Y)$ be metric spaces. Then a map $f\colon X\to Y$ is a coarse equivalence if $f(X)$ is $C$-dense in $Y$ for some constant $C$ and
\[d_X(x_n,y_n)\to +\infty \Longleftrightarrow d_Y(f(x_n),f(y_n))\to +\infty\]
for any two sequences $(x_n)_n$ and $(y_n)_n$ in $X$.
\end{definition}
An interesting way of thinking about a coarse equivalence between box spaces follows from Lemma 1 of \cite{AA}.
\begin{lemma}[Khukhro, Valette]\label{AA1}
Let $\Phi\colon \Box_{N_i}G\to \Box_{M_i}H$ be a coarse equivalence between the box spaces of the residually finite finitely generated groups $G$ and $H$. Then there exists a constant $A$ and an almost permutation $\phi$ between the components of $\Box_{N_i}G$ and the components of $\Box_{M_i}H$ such that $\Phi\vert_{G/N_i}$ is an $(A,A)$-quasi-isometry between $G/N_i$ and $\phi(G/N_i)$.
\end{lemma}
Note that an almost permutation between sets $A$ and $B$ is a bijection between a co-finite subset of $A$ and a co-finite subset of $B$.\\
Since $G/N_i$ and $\phi(G/N_i)$ are $(A,A)$-quasi-isometric we have that both the diameter and the order of $G/N_i$ and $\phi(G/N_i)$ are quite similar.
In \cref{F} we will compare the full box spaces of the free groups, using the similarity of the order, i.e. we will compare the normal subgroup growth of different free groups.
\begin{theorem}\label{free}
Let $2\le k \le d$ with $2(k+1)<\frac{(d-1)^2}{4d}$. Then $\Boxf F_d$ is not coarsely equivalent to $\Boxf F_k$.
\end{theorem}
In \cref{Z} we will use the similarity of diameter to prove that $\Boxf\Z^n$ is not coarsely equivalent to the full box space of a $2$-generated group.
\begin{theorem}\label{mainZn}
Let {\Alln} and let $H$ be a $2$-generated group. Then $\Boxf H$ is not coarsely equivalent to $\Boxf \Z^n$.
\end{theorem}
The most notable thing about the proofs of these theorems is that we do not show that the components of the box spaces are different, we show that one of the box spaces has too many small components compared to the other. In other words, it is still open whether for every box space $\Box_{N_i} F_d$, there exists a box space $\Box_{M_i} F_k$, which is coarsely equivalent to $\Box_{N_i} F_d$.
\\\\
The author would like to thank Ana Khukhro and Alain Valette for the many useful comments.

\section{The full box spaces of the free groups}\label{F}
In this section we will prove that the full box spaces of free groups are different, at least if the amount of generators is sufficiently different. This suggests that the full box spaces of all free groups are different.\\
In the proof we make use of normal subgroup growth, for further reading on (normal) subgroup growth we refer to \cite{lub}. For this paper we only need $a_n^\lhd(G)$, which is the amount of normal subgroups of the group $G$ of index $n$.

\begin{proof}[Proof of \cref{free}]
	Suppose that the full box spaces of the free groups $F_d$ and $F_k$ are coarsely equivalent where $2(k+1)<\frac{(d-1)^2}{4d}$, i.e.\ there is a coarse equivalence $\Phi$ between $\Boxf F_d$ and $\Boxf F_k$. Due to \cref{AA1} there is an almost permutation $\phi$ between the components of $\Boxf F_d$ and the components of  $\Boxf F_k$. As there is some $C'$ such that $\Im\Phi$ is $C'$-dense, components of order less than some $n$ must be mapped to a component of order less than $n\cdot|B[0,C']|$, where $B[0,C']$ is the closed ball of radius $C'$. Now set $C=|B[0,C']|$ and set $D$ equal to the number of components that are not in the domain of $\phi$.\\
	So $|\{N\lhd F_d\mid \#(F_d/N)\le n\}|-D$ is not greater than $|\{N\lhd F_k\mid \#(F_k/N)\le Cn\}|$ for any $n$.
	Note that \[\{N\lhd F_d\mid \#(F_d/N)\le n\}=\{N\lhd F_d\mid [F_d:N]\le n\}=\displaystyle\sum_{i=1}^{n}a_i^\lhd(F_d),\]
	so we find the following inequality:
	\[\displaystyle\sum_{i=1}^{Cn}a_i^\lhd(F_k)+D\ge\displaystyle\sum_{i=1}^{n}a_i^\lhd(F_d)\ge a_n^\lhd(F_d)\]
	It suffices to find an $n$ for which this is not the case.\\
	Let $n$ be a power of $2$, $n=2^m$. Then $a_n^\lhd(F_d)\ge 2^{cm^2}$ if $c<\frac{(d-1)^2}{4d}$ due to Theorem 3.7 of \cite{lub}. As $\frac{(d-1)^2}{4d}>2(k+1)$ we can take $c=2(k+1)+2\delta$, where $\delta>0$. Due to Theorem 2.6 and Lemma 2.5 of \cite{lub}, $a_i^\lhd(F_k)\le i^ki^{2(k+1)\log_2(i)}$ for every $i\in\N$.
	By combining these two bounds we can make the following computation:
	\begin{eqnarray*}
	2^{cm^2}-D & \le & a_n^\lhd(F_d)-D\\
	& \le & \sum_{i=1}^{Cn}a_i^\lhd(F_k)\\
	& \le & \sum_{i=1}^{Cn}i^ki^{2(k+1)\log_2(i)}\\
	& \le & \sum_{i=1}^{Cn}(Cn)^k(Cn)^{2(k+1)\log_2(Cn)}\\
	& = & C^{k+1}n^{k+1}C^{2(k+1)(\log_2(n)+\log_2(C))}n^{2(k+1)(\log_2(n)+\log_2(C))}\\
	& = & 2^{(k+1)\log_2(C)}2^{m(k+1)}2^{2(k+1)(m+\log_2(C))\log_2(C)}2^{2m(k+1)(m+\log_2(C))}\\
	& = & 2^{(k+1)\log_2(C)+m(k+1)+2(k+1)(m+\log_2(C))\log_2(C)+2m(k+1)(m+\log_2(C))}\\
	& = & 2^{2(k+1)m^2+m(k+1)+4m(k+1)\log_2(C)+2(k+1)\log_2(C)^2+(k+1)\log_2(C)}\\
	& = & 2^{(k+1)(2m^2+m+4m\log_2(C)+2\log_2(C)^2+\log_2(C))}\\
	\end{eqnarray*}
	Now we can take $m\gg0$ such that $2m^2+m+4m\log_2(C)+2\log_2(C)^2+\log_2(C)\le(2+\frac{\delta}{k+1})m^2$ and $D<2^{cm^2}-2^{(2(k+1)+\delta)m^2}$. But then we find the following contradiction.
	\begin{eqnarray*}
		2^{cm^2}-D & \le & 2^{(k+1)(2m^2+m+4m\log_2(C)+2\log_2(C)^2+\log_2(C))}\\
		& \le & 2^{(k+1)(2+\frac{\delta}{k+1})m^2}\\
		& = & 2^{(2(k+1)+\delta)m^2}\\
		& < & 2^{cm^2}-D \\
	\end{eqnarray*}
	This proves that $\Boxf F_d$ is not coarsely equivalent to $\Boxf F_k$ for $2(k+1)<\frac{(d-1)^2}{4d}$.
\end{proof}
To use \cref{free} we only need to find appropriate values for $k$ and $d$. The condition $2(k+1)<\frac{(d-1)^2}{4d}$ is satisfied if and only if $d$ is not smaller than $8k+10$. For example $\Boxf F_2$ is not coarsely equivalent to $\Boxf F_{26}$.

\section{The full box spaces of $\Z^n$}\label{Z}
In this section we will prove that the full box space of $\Z^n$ is not coarsely equivalent to the full box space of a $2$-generated group for every \Alln. To do so we will compare the growth in $k$ of $\#\{$quotients with diameter $\le k\}$, which we will call the diameter growth of the components of these full box spaces. Note that the term diameter growth is often used to compare the growth of the diameter with that of the index.
Once we know the diameter growth we can compare the full box spaces using the following result:
\begin{proposition}\label{corpoly}
Let $G$ and $H$ be two groups, with $\Boxf G$ coarsely equivalent to $\Boxf H$ and let $a\in\N$. If $\#\{N\lhd G\mid \diam{G/N}\le k\}=\O(k^a)$, then $\#\{N\lhd H\mid \diam{H/N}\le k\}=\O(k^a)$.
\end{proposition}
\begin{proof}
As there exists a coarse equivalence $\Phi\colon\Boxf G\to\Boxf H$ we can use \cref{AA1} to find an almost permutation $\phi$ between the components of $\Boxf G$ and the components of $\Boxf H$ such that $\Phi\vert_{G/N}$ is an $(A,A)$-quasi-isometry between $G/N$ and $\phi(G/N)$, if $G/N$ lies in the domain of $\phi$.
Therefore $\diam{\phi(G/N)}\le A\diam{G/N}+A$.\\
We can take a constant $C$ such that $\#\{N\lhd G\mid \diam{G/N}\le k\}\le Ck^a$ for every $k$.
Now $\phi$ is an almost permutation, so we can define $D=|(\Im\phi)^c|$. Then we can bound $\#\{N\lhd H\mid \diam{H/N}\le k\}$ as follows:
\begin{eqnarray*}
\#\{N\lhd H\mid \diam{H/N}\le k\} & \le & \#\{N\lhd H\mid \diam{H/N}\le k, H/N\in\Im(\phi)\}+D\\
& \le & \#\{N\lhd G\mid \diam{\phi(G/N)}\le k, G/N\in\dom(\phi)\}+D\\
& \le & \#\{N\lhd G\mid \diam{G/N}\le Ak+A^2, G/N\in\dom(\phi)\}+D\\
& \le & \#\{N\lhd G\mid \diam{G/N}\le Ak+A^2\}+D\\
& \le & C(Ak+A^2)^a+D.\\
\end{eqnarray*}
So $\#\{N\lhd H\mid \diam{H/N}\le k\}=\O(k^a)$.
\end{proof}
Now we want to calculate the diameter growth of $\Boxf\Z^n$.

\begin{proposition}\label{growZ}
	For every $n\in\N$ we have \[\#\{N\lhd\Z^n\mid \diam{\Z^n/N}\le k \}=\Omega\left(k^{n^2}\right).\]
\end{proposition}
\begin{proof}
	Fix a $k$ and consider the subgroups of $\Z^n$ generated by $x_1,\ldots,x_n$ with $\frac{k}{2n}< x_{ii}\le\frac{k}{n}$ and $|x_{ij}|\le\frac{k}{2n^2}$ for every $i\not= j$, where $x_i=(x_{i1},\ldots,x_{in})$.
	The number of possibilities for $x_1,\ldots,x_n$ is $\left(\frac{k}{2n}\right)^n\left(\frac{2k}{2n^2}+1\right)^{n(n-1)}$. This is more than $\frac{1}{(2n)^n}\frac{1}{n^{2n(n-1)}}k^{n^2}$. So it suffices to show that all these subgroups $N$ are different and the diameter of $\Z/N$ is not greater than $k$.\\
	To show that these subgroups are different take $N=N'$ where $N$ is generated by $x_1,\ldots,x_n$ and $N'$ is generated by $x'_1,\ldots,x'_n$. For every $i\le n$ we can take $x'_i=a_1x_1+\ldots+a_nx_n$ with $a_1,\ldots,a_n\in\Z$, since $N=N'$. Now take $j\not=i$ such that $a_j$ is maximal. By projecting on the $j^\text{th}$-component we get the following:
	\begin{eqnarray*}
	\frac{k}{2n^2} & \ge & |a_1x_{1j}+\ldots+a_nx_{nj}|\\
	& \ge & |a_j|x_{jj}-\sum_{k\not=j}|a_kx_{kj}|\\
	& > & \frac{k}{2n}|a_j|-\sum_{k\not=j}\frac{k}{2n^2}|a_k|\\
	& \ge & \frac{k}{2n}|a_j|-(n-1)\frac{k}{2n^2}|a_j|\\
	& = & \frac{k}{2n^2}|a_j|.
	\end{eqnarray*}
	We can conclude that $a_j=0$, therefore only $a_i$ can be different from $0$, which has to be equal to $1$, because $\frac{k}{2n}<x_{ii},x'_{ii}\le\frac{k}{n}$, so $x_i'=x_i$. This is true for every $i$, so $N$ and $N'$ are generated by the same vectors $x_1,\ldots,x_n$.\\
	To prove that $\diam{\Z^n/N}\le k$ suppose there is such a subgroup $N$ for which $\diam{\Z^n/N}>k$. So there is an element in $\Z^n/N$ such that for every representing vector $y=(y_1,\ldots,y_n)$ in $\Z^n$ we have $\displaystyle\sum_{i=1}^n |y_i|>k$. Let $y$ be the representing vector for which $\|y\|$ is minimal and let $i$ be such that $|y_i|$ is maximal. Without loss of generality we may assume $y_i$ to be positive. Now as $\displaystyle\sum_{i=1}^n |y_i|>k$, we find that $y_i>\frac{k}{n}\ge x_{ii}>0$ and we get the following:
	\begin{eqnarray*}
	\|y-x_i\| & = & \sum_{j=1}^{n} |y_j-x_{ij}|\\
	& \le & y_i-x_{ii} + \sum_{j\not= i}\left( |y_j| + |x_{ij}|\right)\\
	& < & y_i-\frac{k}{2n} + \sum_{j\not= i} \left(|y_j| + \frac{k}{2n^2}\right)\\
	& = & \|y\| - \frac{k}{2n^2}.
	\end{eqnarray*}
	Now $y-x_i$ is a smaller representing vector of the same element as $y$, which is a contradiction.\\
	So for all these subgroups $N$ we have $\diam{\Z^n/N}\le k$, which proves that $\#\{N\lhd\Z^n\mid \diam{\Z^n/N}\le k \}=\Omega\left(k^{n^2}\right)$.
\end{proof}
In the proof of \cref{mainZn} it will suffice to know the diameter growth of a $2$-generated virtually $\Z^n$ group $H$. Consequently $H$ must be $\Z^n$-by-finite, because one can turn the finite index subgroup $\Z^n$ into a finite index normal subgroup by taking the intersection of all $g^{-1}\Z^n g$, which is again $\Z^n$ as it is a finite index subgroup in $\Z^n$.
In order to calculate this growth we will first restrict the normal subgroups of $H$ to the finite index normal subgroup $\Z^n\lhd H$. To better understand these normal subgroups of $\Z^n$ we define {\mingen}s.

\begin{definition}
	A {\mingen} of $N\lhd\Z^n$ is the subset $\{x_1,\ldots,x_n\}$ of $N$ where $x_1$ is the smallest vector in $N$ (for the euclidean norm) and $x_i$ is the smallest vector in $N\setminus\langle x_1,\ldots,x_{i-1}\rangle$ such that
	$N\cap\Span{x_1,\ldots,x_i}=\langle x_1,\ldots,x_i\rangle$.
\end{definition}
A {\mingen} is a generating set of $N$, because it is linearly independent and therefore $N=N\cap\Span{x_1,\ldots,x_n}=\langle x_1,\ldots,x_n\rangle$.\\
Note that such a generating set always exists. Also note that a subset of a {\mingen} is a {\mingen} of what it generates. This notion will be important to control the diameter of $\Z^n/(N\cap\Z^n)$. 

\begin{lemma}\label{bound}
	For every {\AllN} there exists a constant $D_n\in\N$ such that for every subgroup $N$ of $\Z^n$ and every {\mingen}  $\{x_1,\ldots,x_n\}$ we have $\norm{a_1x_1+\ldots+a_nx_n}\ge\frac{1}{D_n}\displaystyle\max_i\{\|a_ix_i\|\}$ for every $a_1,\ldots,a_n\in\R$.
\end{lemma}
As we will do in the proof of \cref{bound}, we define $D_n$ recursively with $D_1=1$ and $D_n=D_{n-1}^2(4n^2D_{n-1}^3)^n$.

If the {\mingen} we choose happens to be orthogonal, this lemma would be obvious. The main idea behind the proof is to show that {\mingen}s are sufficiently similar to being orthogonal. In the proof we will assume that \cref{bound} is true up to some value $n$. We will use this to prove an intermediate result (\cref{orth} for $m=n$) and then we will use that to show that \cref{bound} is true for $n+1$.
\begin{lemma}\label{orth}
	Let $\{x_1,\ldots,x_{m+1}\}$ be a {\mingen} and let $p$ be the orthogonal projection on $\Span{x_1,\ldots,x_{m}}$, so we can write $p(x_{m+1})=a_1x_1+\ldots+a_mx_m$. Suppose \cref{bound} is satisfied for all $n\le m$. Then $|a_m|\le\frac{m}{2}D_m$ and $|a_i|\le\frac{m}{2}D_{m}^2$ for all $i<m$.
\end{lemma}
For this lemma we will also assume that $D_{i+1}\ge \frac{i}{2}D_{i}^2+1$ for every $i\ge 1$, which will be the case in the proof of \cref{bound}.
\begin{proof}
	We proceed by contradiction. Let $\{x_1,\ldots,x_{m+1}\}$ be a {\mingen} with the smallest $m$ such that it does not satisfy \cref{orth}. Then we find
	\[\norm{p(x_{m+1})}=\norm{a_1x_1+\ldots+a_mx_m}\ge\frac{1}{D_m}\max_i\norm{a_ix_i}\ge\frac{|a_m|}{D_m}\norm{x_m}.\]
	However as $\{x_1,\ldots,x_{m+1}\}$ is a {\mingen} we have that for every $b_1,\ldots,b_m$ in $\Z$ $\norm{x_{m+1}}\le\norm{x_{m+1}-b_1x_1-\ldots-b_mx_m}$, we even have $\norm{p(x_{m+1})}\le\norm{p(x_{m+1})-b_1x_1-\ldots-b_mx_m}$, because the projections of both vectors onto $\Span{x_1,\ldots,x_m}^\perp$ are equal. If we take $b_i$ such that $|b_i-a_i|\le\frac{1}{2}$, then we find the following inequality:
	\begin{multline*}
	\norm{p(x_{m+1})}\le\norm{p(x_{m+1})-b_1x_1-\ldots-b_mx_m}\\\le\norm{(a_1-b_1)x_1}+\;\ldots\;+\norm{(a_m-b_m)x_m}\le\frac{m}{2}\norm{x_m}.
	\end{multline*}
	Combining these inequalities we conclude that $|a_m|\le\frac{m}{2}D_{m}$. As we assume this {\mingen} does not satisfy the lemma there must be an $a_i$ such that $|a_i|>\frac{m}{2}D_{m}^2$, let $l$ be the largest such $i$.
	\\
	Now let $p_i$ be the orthogonal projection onto $\Span{x_1,\ldots,x_i}$. We will use these projections to bound the corresponding $|a_i|$. We already have $p(x_{m+1})=a_1x_1+\ldots+a_mx_m$. Now we take something similar for the projections $p_i$:
	\begin{eqnarray*}
	p_{m-1}(a_mx_m) & = & a_{m-1,m}x_{m-1}+\ldots+a_{1,m}x_1\\
	p_{m-2}((a_{m-1}+a_{m-1,m})x_{m-1}) & = & a_{m-2,m-1}x_{m-2}+\ldots+a_{1,m-1}x_1\\
	 & \vdots & \\
	p_{l}((a_{l+1}+a_{l+1,m}+\ldots+a_{l+1,l+2})x_{l+1}) & = & a_{l,l+1}x_l+\ldots+a_{1,l+1}x_1\\
	\end{eqnarray*}
	Let $m'$ be such that $l\le m'< m$. As before we have
	\begin{multline*}
	\norm{p_{m'}(x_{m+1})}=\norm{a_1x_1+\ldots+a_{m'}x_{m'}+a_{1,m}x_1+\ldots+a_{m',m}x_{m'}+\ldots+a_{m',m'+1}x_{m'}}\\\ge\frac{1}{D_{m'}}\norm{(a_{m'}+a_{m',m}+\ldots+a_{m',m'+1})x_{m'}}\ge\frac{|a_{m'}+a_{m',m}+\ldots+a_{m',m'+1}|}{D_{m'}}\norm{x_{m'}}.
	\end{multline*}
	As before we can take $b_1,\ldots,b_m$ in $\Z$ such that $|b_i-a_i-a_{i,m}-\ldots-a_{i,m'+1}|\le\frac{1}{2}$ for every $i$.
	Now we find
	\[\norm{p_{m'}(x_{m+1})}\le\norm{p_{m'}(x_{m+1})-b_1x_1-\ldots-b_{m'}x_{m'}}\le\frac{1}{2}\norm{x_1}+\;\ldots\;+\frac{1}{2}\norm{x_{m'}}\le\frac{m'}{2}\norm{x_{m'}}.\]
	So $\frac{m'}{2}D_{m'}\ge|a_{m'}+a_{m',m}+\ldots+a_{m',m'+1}|$.
	As $m$ is assumed to be the smallest value for which this lemma is not true, we have that when $p_{m'}(x_{m'+1})$ is written as a linear combination of $x_1,\ldots,x_{m'}$, where the coefficient of $x_m$ is not greater than $\frac{m'}{2}D_{m'}$ and the other coefficients are not greater than $\frac{m'}{2}D_{m'}^2$. Now as
	\[p_{m'}((a_{m'+1}+a_{m'+1,m}+\ldots+a_{m'+1,m'+2})x_{m'+1}) = a_{m',m'+1}x_{m'}+\ldots+a_{1,m'+1}x_1\]
	we have $|a_{m',m'+1}|\le \frac{m'}{2}D_{m'}|a_{m'+1}+a_{m'+1,m}+\ldots+a_{m'+1,m'+2}| \le\frac{m'}{2}D_{m'}\frac{m'+1}{2}D_{m'+1}$ and $|a_{i,m'+1}|\le \frac{m'}{2}D_{m'}^2|a_{m'+1}+a_{m'+1,m}+\ldots+a_{m'+1,m'+2}| \le\frac{m'}{2}D_{m'}^2\frac{m'+1}{2}D_{m'+1}$ for $i<m'$.\\
	Now we had $\frac{l}{2}D_{l}\ge|a_{l}+a_{l,m}+\ldots+a_{l,m'+1}|$, so using the fact that $D_{i+1}\ge \frac{i}{2}D_{i}^2+1$ and $iD_i\le (i+1)D_{i+1}$ for every $i\ge 1$, we can make the following computation.
	\begin{eqnarray*}
	|a_l| & \le & |a_{l,m}|+\ldots+|a_{l,l+1}|+\frac{l}{2}D_l\\
	& \le & \frac{m-1}{2}D_{m-1}^2\frac{m}{2}D_{m} + \ldots + \frac{l+1}{2}D_{l+1}^2\frac{l+2}{2}D_{l+2}+  \frac{l}{2}D_{l}\frac{l+1}{2}D_{l+1}+\frac{l}{2}D_l\\
	& \le & \frac{m-1}{2}D_{m-1}^2\frac{m}{2}D_{m} + \ldots + \frac{l+1}{2}D_{l+1}^2\frac{l+2}{2}D_{l+2}+  \frac{l+1}{2}D_{l+1}\left(\frac{l}{2}D_{l}+1\right)\\
	& \le & \frac{m-1}{2}D_{m-1}^2\frac{m}{2}D_{m} + \ldots + \frac{l+1}{2}D_{l+1}^2\frac{l+2}{2}D_{l+2}+  \frac{l+1}{2}D_{l+1}^2\\
	& \le & \frac{m-1}{2}D_{m-1}^2\frac{m}{2}D_{m} + \ldots + \frac{l+1}{2}D_{l+1}^2\frac{l+2}{2}D_{l+2}+  \frac{l+2}{2}D_{l+2}\\
	& \vdots & \\
	& \le & \frac{m-1}{2}D_{m-1}^2\frac{m}{2}D_{m} +  \frac{m}{2}D_{m}\\
	& \le & \left(\frac{m-1}{2}D_{m-1}^2 + 1\right) \frac{m}{2}D_{m}\\
	& \le & \frac{m}{2}D_{m}^2
	\end{eqnarray*}
	But we assumed $|a_l|>\frac{m}{2}D_{m}^2$, and so we have a contradiction, which proves this lemma.
\end{proof}
Now we can use this result to prove \cref{bound}.
\begin{proof}[Proof of \cref{bound}]
	We define $D_n$ recursively with $D_1=1$ and $D_n=D_{n-1}^2\left(4n^2D_{n-1}^3\right)^n$. For every subgroup $N\lhd\Z^n$ we can take a {\mingen} $x_1,\ldots,x_n$.\\
	Let $n$ be the smallest value for which the lemma is not true, i.e.\ there exist $a_i$ such that  $\frac{1}{D_n}\displaystyle\max_i\{\|a_ix_i\|\} > \norm{a_1x_1+\ldots+a_nx_n}$. As the lemma is obvious for $n=1$, we may assume that $n\ge 2$.
	\\
	First we observe that $\norm{a_ix_i}$ must be similar for all $i$, that is ${\displaystyle\min_i}\{\|a_ix_i\|\}>\frac{1}{2D_{n-1}}{\displaystyle\max_i\{\|a_ix_i\|\}}$.
	We can see this by combining the reverse triangular inequality with the fact that a subset of a {\mingen} is a {\mingen} of what it generates: $\frac{1}{D_n}{\displaystyle\max_i\{\|a_ix_i\|\}} > \norm{a_1x_1+\ldots+a_nx_n} \ge \frac{1}{D_{n-1}}\displaystyle\max_i\{\|a_ix_i\|\}-\min_i\{\|a_ix_i\|\}.$
	So we get the desired result that ${\displaystyle\min_i}\{\|a_ix_i\|\}> \left(\frac{1}{D_{n-1}}-\frac{1}{D_n}\right){\displaystyle\max_i\{\|a_ix_i\|\}} \ge \frac{1}{2D_{n-1}}{\displaystyle\max_i\{\|a_ix_i\|\}}$.
	\\\\
    To continue we would prefer for $x_n$ to be orthogonal to $\Span{x_1,\ldots,x_{n-1}}$. However a partial result will suffice.
	We will show that the angle between $x_n$ and the span of $x_1,\ldots,x_{n-1}$ can not be arbitrarily small, which will prove the lemma.
	So let $p$ be the orthogonal projection onto $\Span{x_1,\ldots,x_{n-1}}$.
	\\\\
	Now distinguish two cases according to whether or not $nD_{n-1}^2|a_n|$ is greater or smaller than $\displaystyle\max_i\{|a_i|\}$.\\\\
	Suppose $\displaystyle\max_i\{|a_i|\}>nD_{n-1}^2|a_n|$.
	As such we can write $p(x_n)$ as the linear combination $a_1'x_1+\ldots +a_{n-1}'x_{n-1}$.
	\\
	Due to \cref{orth} we know that $|a_i'|\le\frac{n}{2}D_{n-1}^2$ for every $i$.
	Now we can take $k$ such that $|a_k|$ is maximized. By combining $|a_k'|\le\frac{n}{2}D_{n-1}^2$ with $\displaystyle\max_i\{|a_i|\}=|a_k|>nD_{n-1}^2|a_n|$ we find that $|a_k+a_k'a_n|\ge|a_k|-\frac{nD_{n-1}^2}{2}|a_n|\ge\frac{1}{2}|a_k|$. This admits the following computation:
	\begin{eqnarray*}
	\frac{1}{D_n}\max_i\{\|a_ix_i\|\} & \ge & \norm{a_1x_1+\ldots+a_nx_n}\\
	& \ge & \norm{p(a_1x_1+\ldots+a_{n-1}x_{n-1}+a_nx_n)}\\
	& \ge & \norm{a_1x_1+\ldots+a_{n-1}x_{n-1}+a_np(x_n)}\\
	& \ge & \norm{(a_1+a_1'a_n)x_1+\ldots+(a_{n-1}+a_{n-1}'a_n)x_{n-1}}\\
	& \ge & \frac{1}{D_{n-1}}\displaystyle\max_i\{\|(a_i+a_i'a_n)x_i\|\}\\
	& \ge & \frac{1}{2D_{n-1}}\|a_kx_k\|\\
	& \ge & \frac{1}{2D_{n-1}}\displaystyle\min_i\{\|a_ix_i\|\}\\
	& \ge & \frac{1}{4D_{n-1}^2}\max_i\{\|a_ix_i\|\}.\\
	\end{eqnarray*}
	Now $n\ge 2$, so $D_n=2D_{n-1}\left(2n^2D_{n-1}\right)^n>4D_{n-1}^2$, which contradicts the earlier computations.
	\\\\
	Up to this point we essentially only used that $x_n$ can not be shortened by adding a linear combination $\lambda_1x_1+\ldots+\lambda_{n-1}x_{n-1}$ with $\lambda_1,\ldots,\lambda_{n-1}\in\Z$. However if $\displaystyle\max_i\{|a_i|\}\le nD_{n-1}^2|a_n|$ this will not be possible. For example for every $\varepsilon>0$ we have $(2,0,0,0,0)$, $(0,2,0,0,0)$, $(0,0,2,0,0)$, $(0,0,0,2,0)$, $(1,1,1,1,\varepsilon)$, but the group generated by these vectors contains $(0,0,0,0,2\varepsilon)$. In the continuation of this proof we will look for a vector like $(0,0,0,0,2\varepsilon)$, more precisely a short vector that is almost orthogonal to $x_1,\ldots, x_{n-1}$.
	\\\\
	As $\displaystyle\max_i\{|a_i|\}\le nD_{n-1}^2|a_n|$, we have $nD_{n-1}^2\norm{a_nx_n}\ge\displaystyle\max_i{\norm{a_ix_i}}$, as $x_n$ is the biggest vector in the basis $\{x_1,\ldots,x_n\}$. Let $e$ be a unit vector perpendicular to $\Span{x_1,\ldots,x_{n-1}}$. Then $\frac{nD_{n-1}^2}{D_n}\norm{a_nx_n} \ge \norm{a_1x_1+\ldots+a_nx_n} \ge |a_nx_n\cdot e|$, so $\norm{x_n}\ge \frac{D_n}{nD_{n-1}^2}|x_n\cdot e|$.\\
	Now for every $m\in\N$ we can take $p(mx_n)=b_1x_1+\ldots +b_{n-1}x_{n-1}+c_1x_1+\ldots +c_{n-1}x_{n-1}$ with $b_i\in\Z$ and $|c_i|\le\frac{1}{2}$ for every $i$.
	What we are looking for is an $m$ such that $c_1,\ldots,c_{n-1}$ are close to zero. In that case $p(mx_n - b_1x_1 - \ldots - b_{n-1}x_{n-1})$ is small.\\
	To make this precise: for every $i<n$ there exists a $k_i\in\N$ such that $c_i\in\left[\frac{k_i}{4n^2D_{n-1}^3},\frac{k_i+1}{4n^2D_{n-1}^3}\right]$, with $k_i$ between $-2n^2D_{n-1}^3$ and $2n^2D_{n-1}^3-1$. Now due to the pigeonhole principle there will be an $m,m'\le (4n^2D_{n-1}^3)^{n-1}$ with $k_i=k_i'$ for every $i$.
	Now $(m-m')x_n$ will be the vector we are looking for, because $c_i'-c_i\in\left[\frac{-1}{4n^2D_{n-1}^3},\frac{1}{4n^2D_{n-1}^3}\right]$.  As $x_1$ is the smallest vector in $N$ we can make the following computation:
	\begin{eqnarray*}
	\norm{x_1}^2 & \le & \norm{(b_1-b_1')x_1+\ldots+(b_{n-1}-b_{n-1}')x_{n-1}+(m'-m)x_n}^2\\
	& = & \norm{(b_1-b_1')x_1+\ldots+(b_{n-1}-b_{n-1}')x_{n-1}+p(m'x_n)-p(mx_n)}^2 + |m'-m|^2|x_n\cdot e|^2\\
	& \le & \left(\sum_{i=1}^{n-1}\norm{(c_i'-c_i)x_i}\right)^2 + (4n^2D_{n-1}^3)^{2n-2}\frac{n^2D_{n-1}^4}{D_n^2}\norm{x_n}^2\\
	& \le & \left(\sum_{i=1}^{n-1}\frac{\norm{x_i}}{4n^2D_{n-1}^3}\right)^2 + \left(\frac{(4n^2D_{n-1}^3)^n}{4nD_{n-1}D_n}\right)^2\norm{x_n}^2\\
	& \le & \left(\frac{(n-1)\norm{x_n}}{4n^2D_{n-1}^3}\right)^2 + \frac{1}{8n^2D_{n-1}^6}\norm{x_n}^2\\
	& < & \frac{1}{4n^2D_{n-1}^6}\norm{x_n}^2\\
	\end{eqnarray*}
	However, this contradicts the earlier results that $\displaystyle\max_{i}|a_i|\le 2nD_{n-1}^2$ and $\displaystyle\min_i\norm{a_ix_i}\ge\norm{a_nx_n}$ \[\norm{x_1}=\frac{1}{|a_1|}\norm{a_1x_1}\ge\frac{1}{nD_{n-1}^2|a_n|}\displaystyle\min_i\norm{a_ix_i}\ge\frac{1}{2nD_{n-1}^3}\frac{1}{|a_n|}\norm{a_nx_n}\ge\frac{1}{2nD_{n-1}^3}\norm{x_n}.\]
	So for every $a_1,\ldots,a_n\in\R$ we have $\norm{a_1x_1+\ldots+a_nx_n} \ge \frac{1}{D_n}\displaystyle\max_i\{\|a_ix_i\|\}$.
\end{proof}
As mentioned earlier this lemma will help us to control the diameter of $\Z^n/(N\cap\Z^n)$. However we need to control the diameter of $H/N$. We will show that $H$ must be $\Z^n$-by-finite and then we will consider $\Z^n/(N\cap \Z^n)$. While it is not necessarily true that $\diam{H/N)}\ge\diam{\Z^n/(N\cap \Z^n)}$, it is true up to a constant.

\begin{lemma}\label{diam}
	Let $G$ and $H$ be two groups such that $H$ is $G$-by-finite. Then there exists a constant $C$ such that $C\diam{H/N)}\ge\diam{G/(N\cap G)}$ for every $N\lhd H$.
\end{lemma}
\begin{proof}
	Due to Proposition 2 of \cite{Ana}, there exists a $C'$ such that $\diam[G]{G/(N\cap G)}\le C'\diam[H]{G/(N\cap G)}$ for every $N\in H$. So it suffices to show that there exists a $C$ such that  $\diam[H]{G/(N\cap G)}\le C\diam{H/N}$. As $H$ is $G$-by-finite we can take $F=H/G$ finite and set $C=3|F|$.\\
	We can take $g\in G$ such that $|g|_H=|gN|_{H/N}=\diam[H]{G/(N\cap G)}$. Now take a path between $1$ and $g$ and take $1=b_0,b_1,\ldots,b_{|F|}=g$ on this path with $d_H(b_i,b_{i+1})\ge\left\lceil\frac{|g|_H}{|F|}\right\rceil$. Then for every $i$ there exists an $n_i\in N$ such that $d_H(b_i,n_i)\le\diam{H/N}$. As we have $|F|+1$ elements $n_i$, there will be two indices $i<j$ such that $n_i$ and $n_j$ lie in the same coset of $G$. So there exists an $x\in G\cap N$ such that $n_j=xn_i$. Now we can make the following computation:
	\begin{eqnarray*}
	|g|_H & \le & d_H(x,g)\\
	& \le & d_H(x,xb_i) + d(xb_i,xn_i) + d(n_j,b_j) + d(b_j,g)\\
	& \le & d_H(1,b_i) + d(b_j,g) + d(b_i,n_i) + d(n_j,b_j)\\
	& \le & |g|_H - d_H(b_i,b_j) + d(b_i,n_i) + d(n_j,b_j)\\
	& \le & |g|_H - \left\lceil\frac{|g|_H}{|F|}\right\rceil + 2\diam{H/N}
	\end{eqnarray*}
	So $\displaystyle\frac{|g|_H}{|F|}\le 2\diam{H/N}+1\le 3\diam{H/N}$, which proves the lemma.
\end{proof}
Now we can calculate the amount of intersections $N\cap\Z^n$ we can have such that $\diam{H/N}\le k$.
\begin{lemma}\label{growH}
	Let $H$ be $2$-generated and $\Z^n$-by-finite with \Alln. Then $\#\{N\cap\Z^n\mid N\lhd H, \diam{H/N}\le k\}=\O\left(k^{n^2-1}\right)$.
\end{lemma}
\begin{proof}
	Due to \cref{diam} it suffices to show that $\#\{N\cap\Z^n\mid N\lhd H, \diam{\Z^n/(N\cap\Z^n)}\le k\}=\O\left(k^{n^2-1}\right)$. So take $N\lhd H$ such that $\diam{\Z^n/(N\cap\Z^n)}\le k$. Then we can take $N\cap\Z^n$ generated by $\{x_1,\ldots,x_n\}$ as in \cref{bound} and without loss of generality we can assume $\norm{x_1}\ge\ldots\ge\norm{x_n}$. Now for every vector $x\in\R^n$ we have $d(x,\Z^n)\le\frac{\sqrt{n}}{2}$, in particular we have $d\left(\frac{x_1}{2},\Z^n\right)\le\frac{\sqrt{n}}{2}$. So we can make the following computation:
	\begin{align*}
	k + \frac{\sqrt{n}}{2} & \ge \, \diam{\Z^n/(N\cap\Z^n)} + \frac{\sqrt{n}}{2}\\
	& \ge \, d\left(\frac{x_1}{2}+N,0+N\right)\\
	& = \, \inf_{a_1,\ldots,a_{n}\in\Z}\left(\norm{\left(\frac{1}{2}+a_1\right)x_1+a_2x_2 +\ldots+a_nx_n}\right)\\
	& \ge \, \frac{1}{D_n}\inf_{a_1}\left(\left|\frac{1}{2}+a_1\right|\,\norm{x_1}\right) & \text{by \cref{bound}}\\
	& = \, \frac{1}{2D_n}\norm{x_1}.
	\end{align*}
	We can conclude that  $2D_nk+D_n\sqrt{n}\ge\norm{x_1}\ge\ldots\ge\norm{x_n}$. So for any $i$ we have that $x_i$ lies within $\left[-D_n(2k+\sqrt{n}),D_n(2k+\sqrt{n})\right]^n$.\\
	As $H$ is $2$-generated there is an $\alpha_h\in\Aut(\Z^n)$ different from $\pm\Id$, with $\alpha_h(x)=hxh^{-1}$ where $h\in H$. Note that $\alpha_h$ is of finite order and note that $N\cap\Z^n$ is $\alpha_h$-independent. So there exist $a_i$ such that $\alpha_h(x_n)=a_1x_1+\ldots+a_nx_n$. Note that $\alpha_h$ is an bounded operator on $\R^n$, which allows the following computation:
	\begin{eqnarray*}
	\norm{\alpha_h}\norm{x_n} & \ge & \norm{a_1x_1+\ldots+a_nx_n}\\
	& \ge & \frac{1}{D_n}\max_i\{\norm{a_ix_i}\}\\
	& \ge & \frac{1}{D_n}\max_i\{|a_i|\}\norm{x_n}.
	\end{eqnarray*}
	So $D_n\norm{\alpha_h}\ge\displaystyle\max_i\{|a_i|\}$.\\
	Now we still have to count the different possibilities for $N$. There are fewer of these than the different possibilities for $x_1,\ldots,x_n$, as different subgroups have different generators. Note that every possibility of $x_1,\ldots,x_n$ admits values of $a_1,\ldots,a_n$ associated to $\alpha_h$.
	\\\\
	Now we will show that for any given a sequence $a_1,\ldots,a_n$, the number of $x_1,\ldots,x_n$ satisfying earlier conditions is bounded by $(4D_nk+2D_n\sqrt{n}+1)^{n^2-1}$. As the number of possibilities for any $a_i$ is bounded by $2D_n\norm{\alpha_h}$, the total number of possibilities for $x_1,\ldots,x_n$ is bounded by $(2D_n\norm{\alpha_h})^n(4D_nk+2D_n\sqrt{n}+1)^{n^2-1}=\O(k^{n^2-1})$. These earlier conditions are $D_n\norm{\alpha_h}\ge\displaystyle\max_i\{|a_i|\}$, $2D_nk+D_n\sqrt{n}\ge\norm{x_1}\ge\ldots\ge\norm{x_n}$ and $\alpha_h(x_n)=a_1x_1+\ldots+a_nx_n$.\\
	If there is an $i<n$ such that $a_i\not=0$, then $x_i$ can be deduced from all other $x_j$. So the number of possibilities of $x_1,\ldots,x_n$ is bounded by $\left(4D_nk+2\sqrt{n}D_n+1\right)^{(n-1)n}$.\\
	If for every $i<n$ we have $a_i=0$, then $a_n=\pm1$, because otherwise $\alpha_h$ is not an automorphism. Since $\alpha_h\not=\pm\Id$ we know that $\{x\in\R^n\mid\alpha_h(x)=a_nx\}$ is not the entirety of $\R^n$. Therefore it is at most an $(n-1)$-dimensional subspace of $\R^n$, which reduces the possibilities for $x_n$ to at most  $\left(4D_nk+2D_n\sqrt{n}+1\right)^{n-1}$, while the possibilities of other $x_1,\ldots,x_{n-1}$ is bounded by $\left(4D_nk+2D_n\sqrt{n}+1\right)^{(n-1)n}$. Therefore the total number of possibilities in this case is also bounded by $\left(4D_nk+2D_n\sqrt{n}+1\right)^{n^2-1}$.
	\\\\
	In conclusion we have that for any fixed sequence $a_1,\ldots,a_n$ the number of possibilities of $x_1,\ldots,x_n$ is bounded by $\left(4D_nk+2D_n\sqrt{n}+1\right)^{n^2-1}$.
	So the total number of possibilities for $x_1,\ldots,x_n$ is bounded by $\left(D_n\norm{\alpha_h}\right)^n\left(4D_nk+2D_n\sqrt{n}+1\right)^{n^2-1}$.
	Therefore the possibilities of $N\cap\Z^n$ is bounded by that same number, which means $\#\{N\cap\Z^n\mid N\lhd H, \diam{H/N}\le k\}=\O\left(k^{n^2-1}\right)$.
\end{proof}
Now every intersection $\Z^n\cap N$ can be realized by multiple normal subgroups $N\lhd H$. However this amount is bounded. We give the following improved version of our original proposition, due to Alain Valette.
\begin{proposition}[A.Valette]\label{V}
     let $H$ be a finite group with a normal abelian subgroup $A$ generated by $n\geq 1$ elements, and with index $d=[H:A]$. Let $S(H,A)$ be the set of normal subgroups $N\triangleleft H$ such that $N\cap A=\{1\}$. Then $|S(H,A)|$ is bounded above by a function only depending on $n$ and $d$.
\end{proposition}
\begin{proof}
    Indeed let $\pi: H\rightarrow H/A$ be the quotient map. For $N_1\in S(H,A)$, since $\pi|_{N_1}$ is injective, there are (very crudely) at most $2^d$ possibilities for $\pi(N_1)$. 
    \\
    Now we estimate how many $N_2\in S(H,A)$ are such that $\pi(N_1)=\pi(N_2)$. The subgroup $N_1A$ is isomorphic to the direct product $N_1\times A$, we write its elements as pairs $(n_1,a)$. Now since $N_2A=N_1A$ we may view $N_2$ as the graph of a map $\alpha:N_1\rightarrow A$ (we identify $\pi$ on $N_1\times A$ with the projection on the first factor). So we write $N_2=\{(g,\alpha(g)): g\in N_1\}$ and $N_1\rightarrow N_2: g\mapsto(g,\alpha(g))$ is an isomorphism.
    \\
    Fixing $g\in N_1$, we estimate the number of possibilities for $\alpha(g)$. Since $g^d=1$, we must have $\alpha(g)^d=1$ in $A$. So we must bound $d$-torsion in $A$. 
    \\
    By the theory of elementary divisors, there exist integers $f_1,\ldots,f_k$, with $f_i|f_{i+1}$, such that $A\simeq \bigoplus_{i=1}^k \mathbb{Z}/f_i\mathbb{Z}.$ We have $k\leq n$ as $A$ is $n$-generated. Now there are at most $d$ elements of $d$-torsion in a cyclic group (by uniqueness of subgroups). So there are at most $d^k$ elements of $d$-torsion in $A$. So the number of possibilities for $\alpha(g)$ is at most $d^k\leq d^n$.
    \\
    Therefore the number of possibilities for $N_2$ is at most $(d^n)^{|N_2|}\leq d^{nd}$. Finally we have $|S(H,A)|\leq 2^d\cdot d^{nd}$.
\end{proof}
\begin{corollary}\label{growext}
	Let $H$ be $\Z^n$-by-finite for some \Alln. Then there exists a $C>0$ such that for every $\mathcal{N}\lhd\Z^n$ of finite index the set $\#\{N\lhd H\mid N\cap\Z^n=\mathcal{N}\}\le C$.
\end{corollary}
This is an easy consequence of \cref{V} as $\#\{N\lhd H\mid N\cap\Z^n=\mathcal{N}\} = |S(H/\mathcal{N},\Z^n/\mathcal{N})|$.
Now combining \cref{growH} and \cref{growext} we can control the diameter growth of $H$, which suffices to prove \cref{mainZn}.
\\
In the proof of \cref{mainZn} we will use a generalized version of Theorem 7 of \cite{AA}. We will essentially find two coarsely equivalent sequences of groups that each converge to a group in the space of marked groups. Now by combining \cref{AA1} and Proposition 3 in \cite{AA} we find that these two groups are quasi-isometric.

\begin{proof}[Proof of \cref{mainZn}]
	Suppose there is a coarse equivalence $\Phi$ between $\Boxf H$ and $\Boxf \Z^n$, with $H$ $2$-generated.\\
	We may assume that $H$ is residually finite, because if $H$ is not residually finite, i.e. $\displaystyle\bigcap_{N\lhd H}N\not=\{1\}$, then $\Boxf H = \Boxf H/\displaystyle\bigcap_{N\lhd H}N$ and $H/\displaystyle\bigcap_{N\lhd H}N$ is residually finite. Note that $H$ is still $2$-generated.\\
	Now due to \cref{AA1} there is an almost permutation $\phi$ between the components of $\Boxf H$ and the components of  $\Boxf \Z^n$, where $\Phi|_{X}$ is a quasi-isometry between $X$ and $\phi(X)$ for every component $X$ of $\Boxf H$ in the domain of $\phi$. Since $H$ is residually finite, there is a box space $\Box_{(N_k)}H$ contained in $\Boxf H$. Via $\phi$ this corresponds to a subspace $\displaystyle\coprod_{k}\Z^n/M_k$ of $\Boxf \Z^n$. Now this sequence $\left(\Z^n/M_k\right)_k$ has a subsequence that is constant on bigger and bigger balls, i.e.\ there exists a sequence $k_r$ such that $k_r\to\infty$ as $r\to\infty$ and for every $k,k'\ge k_r$ in this subsequence we have $M_k\cap B[1,r]=M_{k'}\cap B[1,r]$. Now due to a generalized version of Theorem 7 of \cite{AA} $H$ is quasi-isometric a quotient of $\Z^n$, because the intersection of the subsequence $M_k$ converges to a normal subgroup of $\Z^n$. So $H$ is virtually $\Z^m$ with $m\le n$, due to the quasi-isometric rigidity of $\Z^m$.\\
	Due to \cref{growH} we have $\#\{N\cap\Z^n\mid N\lhd H, \diam{H/N}\le k\}=\O\left(k^{m^2-1}\right)$ and due to \cref{growext} we have $\#\{N\lhd H\mid \diam{H/N}\le k\}=\O\left(k^{m^2-1}\right)$. However due to \cref{corpoly} we have that $\#\{N\lhd\Z^n\mid \diam{\Z^n/N}\le k \}=\O\left(k^{m^2-1}\right)$, but as $m\le n$ this is in contradiction with \cref{growZ}.
\end{proof}

\bibliographystyle{alpha}
\bibliography{bib.bbl}

\begin{thebibliography}{Khu12}

\bibitem[Khu12]{Ana}
A.~Khukhro.
\newblock {Box spaces, group extensions and coarse embeddings into Hilbert
  space}.
\newblock {\em Journal of Functional Analysis}, 263(1):115--128, 2012.

\bibitem[KV15]{AA}
Ana Khukhro and Alain Valette.
\newblock Expanders and box spaces.
\newblock {\em arXiv preprint arXiv:1509.01394}, 2015.

\bibitem[LS12]{lub}
Alexander Lubotzky and Dan Segal.
\newblock {\em Subgroup growth}, volume 212.
\newblock Birkh{\"a}user, 2012.

\end{thebibliography}
\vspace{\fill}
\textbf{Author's address:\\}
Institut de Math\'ematiques\\
UniMail\\
11 Rue Emile Argand\\
CH-2000 Neuch\^atel - SWITZERLAND\\
\href{mailto:thiebout.delabie@unine.ch}{thiebout.delabie@unine.ch}

\end{document}